\documentclass[12pt]{amsart}
\usepackage{verbatim, amsmath,amsthm,amssymb, amscd, amsfonts}
\newtheorem{theorem}{Theorem}[section]
\newtheorem{lemma}[theorem]{Lemma}
\newtheorem{proposition}[theorem]{Proposition}
\theoremstyle{definition}
\newtheorem{definition}[theorem]{Definition}

\newtheorem{rema}[theorem]{Remark}
\newtheorem{coro}[theorem]{Corollary}

\begin{document}
\newcommand{\nc}{\newcommand}
\nc{\rnc}{\renewcommand} \nc{\nt}{\newtheorem}


\nc{\TitleAuthor}[2]{\nc{\Tt}{#1}%
    \nc{\At}{#2}%
    \maketitle%
}


\nc{\Hom}{\operatorname{Hom}} \nc{\Mor}{\operatorname{Mor}} \nc{\Aut}{\operatorname{Aut}}
\nc{\Ann}{\operatorname{Ann}} \nc{\Ker}{\operatorname{Ker}} \nc{\Trace}{\operatorname{Trace}}
\nc{\Char}{\operatorname{Char}} \nc{\Mod}{\operatorname{Mod}} \nc{\End}{\operatorname{End}}
\nc{\Spec}{\operatorname{Spec}} \nc{\Span}{\operatorname{Span}} \nc{\sgn}{\operatorname{sgn}}
\nc{\Id}{\operatorname{Id}} \nc{\Com}{\operatorname{Com}}

\nc{\nd}{\mbox{$\not|$}} 
\nc{\nci}{\mbox{$\not\subseteq$}}
\nc{\scontainin}{\mbox{$\mbox{}\subseteq\hspace{-1.5ex}\raisebox{-.5ex}{$_\prime$}\hspace*{1.5ex}$}}


\nc{\R}{{\sf R\hspace*{-0.9ex}\rule{0.15ex}%
    {1.5ex}\hspace*{0.9ex}}}
\nc{\N}{{\sf N\hspace*{-1.0ex}\rule{0.15ex}%
    {1.3ex}\hspace*{1.0ex}}}
\nc{\Q}{{\sf Q\hspace*{-1.1ex}\rule{0.15ex}%
       {1.5ex}\hspace*{1.1ex}}}
\nc{\C}{{\sf C\hspace*{-0.9ex}\rule{0.15ex}%
    {1.3ex}\hspace*{0.9ex}}}
\nc{\Z}{\mbox{${\sf Z}\!\!{\sf Z}$}}


\newcommand{\gd}{\delta}
\newcommand{\sub}{\subset}
\newcommand{\cntd}{\subseteq}
\newcommand{\go}{\omega}
\newcommand{\Pa}{P_{a^\nu,1}(U)}
\newcommand{\fx}{f(x)}
\newcommand{\fy}{f(y)}
\newcommand{\gD}{\Delta}
\newcommand{\gl}{\lambda}
\newcommand{\half}{\frac{1}{2}}
\newcommand{\ga}{\alpha}
\newcommand{\gb}{\beta}
\newcommand{\gga}{\gamma}
\newcommand{\ul}{\underline}
\newcommand{\ol}{\overline}
\newcommand{\Lrraro}{\Longrightarrow}
\newcommand{\equi}{\Longleftrightarrow}
\newcommand{\gt}{\theta}
\newcommand{\op}{\oplus}
\newcommand{\Op}{\bigoplus}
\newcommand{\CR}{{\cal R}}
\newcommand{\tr}{\bigtriangleup}
\newcommand{\grr}{\omega_1}
\newcommand{\ben}{\begin{enumerate}}
\newcommand{\een}{\end{enumerate}}
\newcommand{\ndiv}{\not\mid}
\newcommand{\bab}{\bowtie}
\newcommand{\hal}{\leftharpoonup}
\newcommand{\har}{\rightharpoonup}
\newcommand{\ot}{\otimes}
\newcommand{\OT}{\bigotimes}
\newcommand{\bwe}{\bigwedge}
\newcommand{\gep}{\varepsilon}
\newcommand{\gs}{\sigma}
\newcommand{\OO}{_{(1)}}
\newcommand{\TT}{_{(2)}}
\newcommand{\FF}{_{(3)}}
\newcommand{\minus}{^{-1}}
\newcommand{\CV}{\cal V}
\newcommand{\CVs}{\cal{V}_s}
\newcommand{\slp}{U_q(sl_2)'}
\newcommand{\olp}{O_q(SL_2)'}
\newcommand{\slq}{U_q(sl_n)}
\newcommand{\olq}{O_q(SL_n)}
\newcommand{\un}{U_q(sl_n)'}
\newcommand{\on}{O_q(SL_n)'}
\newcommand{\ct}{\centerline}
\newcommand{\bs}{\bigskip}
\newcommand{\qua}{\rm quasitriangular}
\newcommand{\ms}{\medskip}
\newcommand{\noin}{\noindent}
\newcommand{\raro}{\rightarrow}
\newcommand{\alg}{{\rm Alg}}
\newcommand{\rcom}{{\cal M}^H}
\newcommand{\lcom}{\,^H{\cal M}}
\newcommand{\rmod}{\,_R{\cal M}}
\newcommand{\qtilde}{{\tilde Q^n_{\gs}}}
\nc{\e}{\overline{E}} \nc{\K}{\overline{K}} \nc{\gL}{\Lambda}
\newcommand{\tie}{\bowtie}
\newcommand {\h}{\widehat}
\newcommand {\tl}{\tilde}
\newcommand{\tri}{\triangleright}
\nc{\ad}{_{\dot{ad}}}
\nc{\coad}{_{\dot{{\rm coad}}}}
\nc{\ov}{\overline}

\title{Character tables and normal left coideal subalgebras}
\author{Miriam Cohen}
\address {Department of Mathematics\\
Ben Gurion University, Beer Sheva, Israel}
\email {mia@math.bgu.ac.il}
\author{Sara Westreich}
\address{Department of Management\\
Bar-Ilan University,  Ramat-Gan, Israel}
\email{swestric@biu.ac.il}
\thanks {This research was supported by the ISRAEL SCIENCE FOUNDATION, 170-12.}


\date{Apr-2013}

\begin{abstract}
We continue studying properties of semisimple Hopf algebras $H$ over algebraically closed fields of characteristic $0$ resulting from their generalized character tables. We show that  the generalized character table of $H$ reflect normal left coideal subalgebras of $H.$ These  are the Hopf analogues of normal subgroups in the sense that they arise from Hopf quotients.  We apply these ideas to prove Hopf analogues of known results in group theory. Among the rest  we prove  that columns of the character table are orthogonal and that all entries are algebraic integers. We analyze `semi kernels' and their relations to the character table. We prove a full analogue of the Burnside-Brauer theorem for almost cocommutative $H.$ We also prove the Hopf algebras  analogue of the following (Burnside) theorem:
If G is a non-abelian simple group then $\{1\}$ is the only conjugacy class of G which has prime power order.
\end{abstract}
\maketitle
{\footnotesize{We dedicate this paper to the late Bettina Zoeller-Richmond whose joint seminal work with Nichols is at the basis of many of the  results in this paper.}}
\section*{introduction}

The basic philosophy behind this paper is   that the analogue of normal subgroups for Hopf algebras are normal left coideal subalgebras, $N,$ rather than the more naive notion of normal Hopf subalgebras.
Such $N$ occur whenever one deals with Hopf quotients.
Specific examples  are the so called left kernels of representations of $H,$ which are analogues of kernels of representations of a group $G,$ (an idea introduced by Burciou). We also introduce semi-kernels, a Hopf algebra analogue of certain normal subgroups of  $G$ which contain such kernels.
In a different direction, trivial for groups, examples of  normal left coideal subalgebras are images of the Drinfeld map for quasitriangular Hopf algebras.

It turns out that there is an intrinsict connection between character tables for semisimple Hopf algebras $H$ over a field of characteristic $0$ and certain normal left coideal subagebras of $H.$  Unlike character tables for groups the generalized character table need not be a square matrix.

\medskip
We apply these ideas to prove among the rest the following (Burnside) theorem:
If G is a non-abelian simple group then $\{1\}$ is the only conjugacy class of G which has prime power order.

\bigskip The paper is organized as follows: In the preliminaries we define conjugacy classes ${\mathfrak C}_j,$ normalized class sums $\eta_j$ and  a generalized character table $(\xi_{ij})$ as follows:

Let $R(H)$ denote the character algebra of $H$ and let $\{\frac{1}{d}\gl=F_0,\dots F_{m-1}\}$ be a complete set of central primitive idempotents of $R(H).$ For each $j,$ let $\{f_{ij}\},\,1\le i\le m_j$ be a full set of primitive orthogonal idempotents of the algebra $F_jR(H).$  Then  $f_{ij}R(H)$ are all isomorphic as right $R(H)$-modules of dimension $m_j.$
The {\bf conjugacy class} ${\mathfrak C}_{ij}$ is defined as:
$${\mathfrak C}_{ij}=\gL\leftharpoonup f_{ij}H^*.$$
For each $j$ choose  a representative for each conjugacy class by taking arbitrarily $i,$ and defining:
  $$f_j=f_{ij}\qquad{\mathfrak C}_j={\mathfrak C}_{ij}.$$

Define also
 $${\mathfrak C}^j= \gL \leftharpoonup F_jH^*=\bigoplus_i {\mathfrak C}_{ij}.$$

We generalize also the notions of  {\bf Class sum}  and of a representative of a conjugacy class as follows:
$$ C_j=\gL\leftharpoonup dF_j,\qquad \eta_j=\frac{1}{\dim (f_jH^*)}C_j.$$
We refer to $\eta_j$ as a  {\bf normalized class sum}.

 \medskip Let $\{\chi_0,\dots,\chi_{n-1}\}$ be the set of all irreducible characters of $H.$ The generalized character table $(\xi_{ij})$ of a semisimple Hopf algebra $H$ over $k$ is given by:
$$\xi_{ij}=\left\langle  \chi_i,\eta_j\right\rangle  ,$$
$0\le i\le n-1,\;0\le j\le m-1.$

\bigskip In \S $2$ we prove that normalized class sums $\eta_j$ are in fact irreducible characters of $R(H).$

\medskip\noin{\bf Theorem \ref{muf}:}  Let $H$ be a semisimple Hopf algebra over $k,\,\{\eta_j\}$  the  normalized class sums of $H.$  Then the irreducible character $\mu_j$ of $R(H)$ corresponding to the irreducible $R(H)$-module $f_jR(H)$ can be identified inside $Z(H)$ as $\eta_j.$

\medskip We use this to prove orthogonality of columns of the generalized character table.

\medskip\noin{\bf Theorem \ref{ortho1}:} Let $H$ be a $d$-dimsnsional semisimple Hopf algebra over $\C,$ then
$$\sum_k\xi_{ki}\ol{\xi_{kj}}=\gd_{ij}\frac{d\dim(f_iR(H))}{\dim(f_iH^*)}.$$

\medskip Another property of the generalized character table, which is known for groups and is of essential importance in proving later the existence of various normal left coideal subalgebras, is the following:

\medskip\noin{\bf Theorem \ref{maxgeneral}:}
Let $H$ be a semisimple Hopf algebra over $\C$ with a character table $(\xi_{ij})$ where $\xi_{ij}=\langle\chi_i,\eta_j\rangle $ and   $f_jR(H)$  an irreducible right $R(H)$ module of dimension $m_j.$  Then

\noin{\rm (i)} Each entry of the $i$-th row of the character table satisfies:
$$|\langle\chi_i,\eta_j\rangle  |\le m_j\langle\chi_i,1\rangle .$$

\noin{\rm (ii)} Equality holds if and only if right multiplication by $\chi_i,\;r_{\chi_i}$ acts on $f_jR(H)$ as $\ga_i\Id_{f_jR(H)},$ where $|\ga_i|=\langle \chi_i,1\rangle.$

\noin{\rm (iii)}  $\langle\chi_i,\eta_j\rangle = m_j\langle\chi_i,1\rangle $ if and only if $r_{\chi_i}$ acts on $f_jR(H)$ as $\langle\chi_i,1\rangle \Id_{f_jR(H)}.$

\medskip

In \S 3 we extend the definition of left kernels to $\go$-left kernels, $\go\in\C,$ as follows:
$${\rm LKer}^{\go}_V=\{h\in H|\sum h_1\ot h_2\cdot v=\go h\ot v\quad\forall v\in V,\;|\go|=1\}.$$

Then ${\rm LKer}^{\go}_V$ is a left coideal stable under the adjoint action of $H$ and $\sum_{\go}\bigoplus {\rm LKer}^{\go}_V$ is a graded normal left coideal subalgebra of $H.$ For groups this boils down to the normal subgroup of $G$ sometimes denoted by $Z(\chi_V)$ (\cite[Def 2.26]{is}).

\medskip We  prove essential relations between the character table  and $\go$-left kernels by using an equivalence relation on characters introduced by \cite[Prop. 18]{nr2}.

\medskip\noin {\bf Theorem \ref{connection3}:}
Let $H$ be a semisimple Hopf algebra over an algebraicly closed field $k$ of characteristic $0,\;\chi_i$ a character associated with the irreducible representation $V_i,$ and $F_j$ a central primitive idempotent of $R(H).$ Then the following are equivalent:

\noin{\rm (i)} $F_j\chi_i=\go_i\langle \chi_i,1\rangle  F_j,\,|\go_i|=1.$

\noin{\rm (ii)} $\langle \chi_i,\eta_j \rangle = \go_im_j\langle \chi_i,1 \rangle,\,|\go_i|=1 .$

\noin{\rm (iii)} ${\mathfrak C}^{j}\subset {\rm LKer}^{\go_i}_{V_i}.$

Moreover, if the equivalent conditions hold then $\go_i$ is a root of unity of order $t,$ where $t$ divides $\dim H.$

\medskip In particular, when $\go=1$ we relate normal left coideal subalgebras which are left kernels to the character table:

\medskip\noin {\bf Theorem \ref{connection}:}
Let $H$ be a semisimple Hopf algebra over an algebraicly closed field $k$ of characteristic $0,\;B$ the Hopf subalgebra of $H^*$ generated by $\chi_i,$ and $F_j$ a central primitive idempotent of $R(H).$ Then the following are equivalent:

\noin{\rm (i)} $F_j\chi_i=\left\langle \chi_i,1 \right\rangle F_j$

\noin{\rm (ii)}  $\left\langle \chi_i,\eta_j \right\rangle =m_j\left\langle \chi_i,1 \right\rangle .$

\noin{\rm (iii)} ${\mathfrak C}^{j}\subset {\rm LKer}_{V_i}.$

\medskip

Just as for groups we say that $\chi_V$ is a faithful character if ${\rm LKer}_V=k1.$   If the character algebra $R(H)$ is assumed to be commutative we get a complete analogue of the  Burnside-Brauer theorem as follows. Define the value set
$$\chi(H)=\{\langle  \chi,\eta_j\rangle  \,|\;\text{all }\eta_j\}.$$
Then:

\medskip\noin {\bf Theorem \ref{burnside}:}
 Let $H$ be a semisimple Hopf algebra such that $R(H)$ is commutative and let $\chi$ be a faithful character of $H.$ Set $$t= \text{number of distinct values in }\, \chi(H).$$
 Then each character $\mu$ of $H$ appears with positive multiplicity in at least one of $\{\gep,\chi,\chi^2,\dots,\chi^{t-1}\}.$

\bigskip In \S 4 we apply the results of the previous sections with known results for groups. In particular we  prove a theorem of Burnside for groups in our context.

\medskip\noin{\bf Theorem \ref{isa}:}
Let $H$ be a non-commutative quasitriangular semisimple Hopf algebra over $\C$ whose only normal left coideal subalgebras are $\C$ and $H.$ Then $\C$ is the only  conjugacy class of $H$ with  a prime power dimension.

\medskip This theorem is an essential step in the proof of Burnside $p^aq^b$ theorem for groups, which was generalized to fusion categories in \cite{eno}. We prove here an interesting property of normal left coideal subalgebras of factorizable Hopf algebras of dimension $p^aq^b.$

\medskip\noin{\bf Theorem \ref{hopfsub}:}
Let $H$ is a factorizable semisimple  Hopf algebra over $\C$  of dimension $p^aq^b,\,a+b>0,$ and let $N$ be a  normal left  coidal subalgebra of $H.$ Then $N$ contains a central grouplike element. In particular, any minimal normal left  coidal subalgebra of $H$ is a central Hopf subalgebra of $H$ of  prime order.

\section{Preliminaries} Throughout this paper,  $H$ is a finite-dimensional Hopf algebra over a field $k,$ in some cases we assume $k=\C.$  We denote by $S$ and $s$ the antipodes of $H$ and $H^*$ respectively and $\gL$ and $\gl$ the left and right integrals of $H$ and $H^*$ respectively so that $\langle  \gl,\gL \rangle   = 1.$
We  say that $H$ is unimodular if $\gL$    is also a right integral (this is the case when $H$ is semisimple). Denote by
 $Z(H)$ the center of $H.$

 Recall that any subbialgebra of $H$ is necessarily a Hopf subalgebra.

The Hopf algebra $H^*$ becomes a right and left $H$-module by the {\it hit} actions $\leftharpoonup$ and $\rightharpoonup$ defined for all $a\in H,\,p\in H^*,$
$$\langle  p\leftharpoonup a,a'\rangle  =\langle  p,aa'\rangle  \qquad \langle  a\rightharpoonup p,a'\rangle  =\langle  p,a'a\rangle  $$
$H$ becomes a left and right  $H^*$ module analogously.

Denote by $_{\dot{ad}} $ the left adjoint action of $H$ on itself, that is, for all $a,h\in H,$
$$h_{\dot{ad}}  a=\sum h_1aS(h_2)$$

A left coideal subalgebra of $H$ is called {\it normal} if it is stable under the left adjoint action of $H.$

Let $D(H)$ denote the Drinfeld double of the Hopf algebra $H.$ It is not hard to see that $H$ is a $D(H)$-module with respect to the left adjoint action of $H$ on itself and  the left    action $\rightharpoondown$ of $H^*$ on $H,$ defined by
$$p\rightharpoondown h=h\leftharpoonup s\minus(p).$$

\medskip
Denote by Coc$(H^*)$ the algebra of cocommutative elements of $H^*$ which equals the set of all $p \in H^*$ so that $\langle  p , ab\rangle   = \langle  p, ba\rangle  $ for all $a,b \in H.$ Denote by $R(H)$ the $k$-span of all irreducible characters.  It is an algebra (called the character algebra) which is contained in Coc$(H^*)$. A necessary and sufficient for them to equal is that $H$ is semisimple.

\medskip Recall \cite{sk}, any left coideal subalgebra $A$ of $H$  contains a left integral $\gL_A.$ Moreover, if $A_1\subset A_2$ are left coideal subalgebras then $A_2$ is free over $A_1.$  This implies in particular that:
\begin{rema}\label{sk1} If $A\ne B$ are left coideal subalgebras of $H$ then $\gL_A\ne\gL_B.$ If $H$ is semisimple, then $A$ is semisimple and $\langle  \gL_A,1\rangle  \ne 0.$ \end{rema}



\medskip

Let $V$ be a finite dimensional left $H$-module  with a character $\chi_V$ and let  $R_V$ be  the coalgebra in $H^*$ generated by $\chi_V.$ Then $R_V$ can be described  explicitly as follows: Let $\gd_V$ be the corresponding right $H^*$-comodule structure map on $V,$ then for a basis  $\{v_1\dots, v_n\}$ of $V,$ we have $\gd_V(v_i)=\sum_k v_k\ot x_{ki},$ and $\chi_V=\sum x_{ii}.$ It follows that
\begin{equation}\label{hv} \gd_V(V)\subset V\ot R_V.\end{equation}
Note that  $R_V$ contains all the irreducible constituent of $\chi_V,$ and thus $V$ is irreducible if and only if $R_V$ is a simple coalgebra.
Note also that for any finite dimensional left $H$-modules $V,\,W,$
$$R_VR_W=R_{V\ot W}.$$

Let $H$ be a semisimple Hopf algebra  $\{V_0,\dots V_{n-1}\}$  a complete set of non-isomorphic irreducible $H$-modules, $V_0=k,$ and  $\{E_0,\dots E_{n-1}\}$ and $\{\chi_0,\dots\chi_{n-1}\}$ be the associated central primitive idempotents and  irreducible characters of $H$ respectively, where $E_0=\gL$ and $\chi_0=\gep.$  Let $\dim V_i=d_i=\langle  \chi_i,1\rangle  ,$ then
$\gl=\chi_H=\sum_{i=0}^{n-1}d_i\chi_i.$ One has (see e.g \cite[Cor.4.6]{sc}):
 \begin{equation}\label{dual1}\langle  \chi_i,E_j\rangle  =\gd_{ij}d_j,\qquad \gL\leftharpoonup\chi_j=\frac{1}{d_j}S(E_j).\end{equation}
In particular, $\{\chi_i\},\,\{\frac{1}{d_j}E_j\}$ are dual bases of $R(H)$  and $Z(H)$ respectively.

Since $H$ is semisimple, we have:
$$\chi_i\chi_j=\sum_lm_{ij}^l\chi_l,$$
where $m_{ij}^l$ are non-negative integers.

For a semisimple Hopf algebra over an algebraically closed field $k$ of characteristic $0$, let  Irr$(H)$ be  the set of irreducible characters of $H.$  Then  any simple subcoalgebra $B_i$ of $H^*$ contains precisely one character that generates $B_i$ as a coalgebra. Since $B=\bigoplus_{i\in I}B_i,$ where each $B_i$ is a simple subcoalgebra of $H^*,$ it follows that $B$ is the coalgebra generated by $B\cap {\rm Irr}(H).$ Also, if $\chi\in B$ then all its irreducible constituents belong to $B$ as well.

In particular, if $B$ is a Hopf subalgeba of $H^*$  then
$$\gl_B=\sum_{\chi_i\in {\rm Irr}(H)\cap B}d_i\chi_i$$
is a nonzero integral for $B.$

\medskip
It is known (\cite{k,z}) that when $H$ is semisimple then so is $R(H).$ Let $\{\frac{1}{d}\gl=F_0,\dots F_{m-1}\}$ be a complete set of central primitive idempotents of $R(H).$ For each $j,$ let $\{f_{ij}\},\,1\le i\le m_j$ be a full set of primitive orthogonal idempotents of the algebra $F_jR(H).$  Then  $f_{ij}R(H)$ are all isomorphic as right $R(H)$-modules of dimension $m_j.$

\medskip Motivated by group theory  the {\bf conjugacy class} ${\mathfrak C}_{ij}$ as:
\begin{equation}\label{cci}{\mathfrak C}_{ij}=\gL\leftharpoonup f_{ij}H^*.\end{equation}
Then we have shown in \cite[Th.1.4]{cw4} that:

\medskip

   {\it ${\mathfrak C}_{ij}$ is an irreducible $D(H)$-module.}

  \medskip
  Moreover, ${\mathfrak C}_{ij}\cong {\mathfrak C}_{tj}$ as left $D(H)$-modules for all $1\le i,t\le m_j.$ In particular, they have the same dimension, which equals $\dim(f_{ij}H^*)$. Choose arbitrarily $i,$ and set
\begin{equation}\label{fj0}f_j=f_{ij}\qquad{\mathfrak C}_j={\mathfrak C}_{ij}.\end{equation}
  We have shown in \cite[Th.1.4]{cw4} that,
\begin{equation}\label{nik} H\cong \oplus_{j=0}^{m-1} {\mathfrak C}_j^{\oplus m_j}\end{equation}
as $D(H)$-modules.

\medskip We generalize also the notions of  {\bf Class sum}  and of a representative of a conjugacy class as follows:
\begin{equation}\label{ci} C_j=\gL\leftharpoonup dF_j,\qquad \eta_j=\frac{1}{\dim (f_jH^*)}C_j.\end{equation}
We refer to $\eta_j$ as a  {\bf normalized class sum}.

By general ring theory considerations (see e.g. \cite[\S 3]{lo}), $\dim f_jH^*$  does not depend on the choice of the primitive orthogonal idempotent $f_j$ belonging to $F_j,$ hence neither does $\eta_j.$ Note that contrary to the group situation, one does not always have $\eta_j\in {\mathfrak C}_j,$ however, $\eta_j\in \bigoplus_{i=1}^{m_i}{\mathfrak C}_{ij}.$

\medskip Since $f_0=\frac{1}{d}\gl,$ it follows that $f_0=F_0,\,m_0=1,\,\dim f_0H^*=1$ and $\eta_0=1.$

 \medskip  When $R(H)$ is commutative then the central primitive idempotents of $R(H),\,\{F_j\}$ are also minimal idempotents, thus $F_j=f_j$  and they form another basis for $R(H).$ Moreover, for any character $\chi$ and primitive idempotent $F_j,$
\begin{equation}\label{almost}  \chi F_j=\langle  \chi,\eta_j\rangle   F_j.\end{equation}

\medskip We can define now a generalized character table for $H$ as follows:
\begin{definition}\label{gct}  The generalized character table $(\xi_{ij})$ of a semisimple Hopf algebra $H$ over $k$ is given by:
$$\xi_{ij}=\left\langle  \chi_i,\eta_j\right\rangle  ,$$
$0\le i\le n-1,\;0\le j\le m-1.$
\end{definition}
Note that $\xi_{i0}=\langle  \chi_i,1\rangle  =d_i$  and $\xi_{0j}=m_j=\dim f_jR(H)$ for all $0\le i \le n-1,\,0\le j\le m-1.$

\section{Some properties of the generalized character table}

Let $H$ be a semisimple Hopf algebra over $\C.$ Nichols and Richmond \cite{nr1,nr2} discussed various properties of $R(H)$ as an inner product space with involution. Their method is described as follows.  For $u=\sum\ga_i\chi_i,\;v=\sum \gb_j\chi_j,$ set
\begin{equation}\label{inner}(u,v)=\sum \ga_i\overline{\gb}_i\end{equation}
Then $(\,,\,)$ is an inner product on $R(H).$ Define an involution
$*$ on characters by
$$\chi^*=s(\chi)$$
and extend it to $R(H)$ as follows. For $u=\sum\ga_i\chi_i\in R(H)$ set
$$u^*=\sum\ol{\ga}_i\chi_i^*$$

As an algebra, $R(H)$ has its own set of characters. By \cite[Remark 11]{nr2}, for all $x\in R(H),\;\mu$ a character defined on $R(H)$ we have:
\begin{equation}\label{mu}\left\langle \mu,x^*\right\rangle =\ol{\left\langle \mu,x\right\rangle }.\end{equation}

\medskip By \eqref{dual1}, $Z(H)$ can be considered as a dual vector space to $R(H),$ but the connection is much deeper.

Recall, as a consequence of the orthogonality of characters given in \cite{la}, $R(H)$ is a symmetric algebra with a symmetric form $\gb$ defined by
$$\gb(p,q)=\left\langle \gL,pq\right\rangle ,$$ where $\gL$ is the corresponding central form  and the Casimir element is
\begin{equation}\label{casimir}\sum_{k=0}^{n-1}\chi_k\ot s(\chi_k).\end{equation}

It follows that for all $p\in R(H),$
\begin{equation}\label{casimir2}
p=\sum_k\langle  \gL,p\chi_k\rangle  s(\chi_k)\end{equation}

\medskip
For $x\in R(H),$ let $l_x (r_x)$ denote the operator of left (right) multiplication of $R(H)$ by $x.$

\begin{theorem}\label{muf}  Let $H$ be a semisimple Hopf algebra over an algebraically closed field of characteristic $0,\,\{\eta_j\}$  the  normalized class sums of $H$ as defined in \eqref{ci}. Then the irreducible character $\mu_j$ of $R(H)$ corresponding to the irreducible right $R(H)$-module $f_jR(H)$ can be identified inside $Z(H)$ as $\eta_j.$
In particular $\langle  \gep,\eta_j\rangle  =m_j.$
\end{theorem}
\begin{proof}

It can be seen  that for all $0\le j\le  m,$
\begin{equation}\label{casimirfi}
\sum_{k=0}^{n-1}\chi_kf_js(\chi_k)=\frac{d}{\dim(f_jH^*)}F_j.\end{equation}
See e.g \cite[Prop.2.1]{cw3} where $\sum_{k=0}^{n-1}\chi_kf_js(\chi_k)$ is denoted there by $\tau_C(f_j).$
By the definition of $\mu_j$ we have:
 $$\left\langle \mu_j,x\right\rangle =Trace(r_x\circ l_{f_j}).$$

By the trace formula for symmetric algebras (see e.g. \cite[Lemma 1.3]{cw3}), we have:
$$\left\langle \mu_j,x\right\rangle =Trace(r_x\circ l_{f_j})=\sum_k\left\langle \gL,\chi_kf_j s(\chi_k)x\right\rangle =\frac{d}{\dim(f_jH^*)}\left\langle \gL,F_jx\right\rangle .$$
The result follows now from the definition of $\eta_j.$
\end{proof}

\begin{rema}\label{iso}
Since $f_jR(H)\cong f_{j'}R(H)$ for any other primitive idempotent $f_{j'}$ of $R(H)$ belonging to  $F_j,$ it follows that $\mu_j$ does not depend on the choice of the primitive idempotent belonging to $F_j$ as was demonstrated in Theorem \ref{muf}.
\end{rema}
As a corollary of Theorem \ref{muf} we obtain,
\begin{coro}\label{algebraic}
All entries of the character table  are algebraic integers.
\end{coro}
\begin{proof}
As proved in \cite{lo}, any character $\chi_i$ satisfies a monic polynomial over $\Z,$  hence so does its image under the representation
$\rho_{f_jR(H)}.$    It follows that all  eigenvalues of $\rho_{f_jR(H)}(\chi_i)$  are algebraic integers and hence so is trace$(\rho_{f_jR(H)}(\chi_i)).$ By Theorem \ref{muf}  trace$(\rho_{f_jR(H)}(\chi_i))=\left\langle \chi_i,\eta_j\right\rangle .$
\end{proof}

As a corollary of Theorem \ref{muf} and \eqref{mu} we obtain:
\begin{coro}\label{adjoin}
When $k=\C,$ then for any normalized class sum in $H,$ and for any element $x\in R(H),$
$$\left\langle x^*,\eta_j\right\rangle =\ol{\left\langle x,\eta_j\right\rangle }.$$
\end{coro}

 We denote $\chi_i^*=s(\chi_i)$ by $\chi_{i^*}.$ Note that since $\gL$ is cocommutative, we have $S(C_j)=S(\gL\leftharpoonup dF_j)=\gL\leftharpoonup dSF_j.$
Since $\dim(F_jH^*)=\dim (H^*F_j)=\dim((SF_j)H^*)$, it follows that $S\eta_j$ is a normalized class sum as well. Denote also $S\eta_j$ by $\eta_{j^*}.$
 Then Corollary \ref{adjoin} implies in particular that
\begin{equation}\label{starj} \xi_{i^*j}=\xi_{ij^*}=\ol{\xi_{ij}}. \end{equation}

We can prove orthogonality relations for the columns of the generalized character table as follows:
\begin{theorem}\label{ortho1} Let $H$ be a $d$-dimsnsional semisimple Hopf algebra over $\C,$ then
$$\sum_k\xi_{ki}\ol{\xi_{kj}}=\gd_{ij}\frac{d\dim(f_iR(H))}{\dim(f_iH^*)}.$$
\end{theorem}
\begin{proof}
\begin{eqnarray*}
\lefteqn{\sum_k\xi_{ki}\ol{\xi_{kj}}=}\\
&=&\sum_k\langle  \chi_k,\eta_i\rangle  \langle  s(\chi_k),\eta_j\rangle  \qquad\text{(by \eqref{starj})} \\
&=& \frac{1}{\dim(H^*f_i)}\langle  \sum_k \langle  \chi_k,(\gL\leftharpoonup dF_i)\rangle   s(\chi_k),\eta_j\rangle  \\
&=&\frac{d}{\dim(H^*f_i)}\langle  \sum_k\langle   F_i\chi_k,\gL\rangle   s(\chi_k),\eta_j\rangle  \\
&=&\frac{d}{\dim(H^*f_i)}\langle  F_i,\frac{1}{\dim(H^*f_j)}(\gL\leftharpoonup dF_j)\rangle  \qquad\text{(by \eqref{casimir2})}\\
&=&\gd_{ij}\frac{d^2}{(\dim(f_iH^*))^2}\langle  \gL,F_i\rangle  \\
&=&\gd_{ij}\frac{d\dim(f_iR(H))}{\dim(H^*f_i)}\qquad\text{(by \cite[(8) and (9)]{cw2})}
\end{eqnarray*}
\end{proof}
\begin{rema}\label{cjcharacter}
When $R(H)$ is commutative then $\dim f_jR(H)=1,$ hence for $i=j$ the right hand side of the orthogonality relations boils down to $\frac{d}{\dim  {\mathfrak C}_j}.$ Thus, as for groups, the dimension of each conjugacy class can be recovered from the character table.
\end{rema}

\medskip Another norm of an element $x\in R(H)$ was defined in \cite{nr2}   as follows:
\begin{equation}\label{normn}
N(x)=\sup_{0\ne u\in R(H)}\frac{\|ux\|}{\|u\|}\end{equation}
where $\|\,\|$ is the norm obtained from the inner product given in \eqref{inner}. It was proved there \cite[Th.14.2]{nr2} that for a character $\chi,$
$$\left\langle \chi,1\right\rangle =N(\chi)$$
This implies that if $u\in R(H)$ is a nonzero eigenvector  of  $r_{\chi}$ with an eigenvalue $\ga,$ then
\begin{equation}\label{ev}|\ga|\le\left\langle \chi,1\right\rangle.\end{equation}


The following lemma about complex numbers will be an essential tool in analyzing   values related to the character table.
\begin{lemma}\label{complex}
Let $\{\ga_k\}_{k=1}^m$ be a set of complex numbers with $\ga_1\ne 0$ and let $\{a_k\}_{k=1}^m$ be a set of non-negative real numbers satisfying:
$$ |\sum_{k=1}^m \ga_k|=\sum_{k=1}^m a_k\quad\text{and}\quad |\ga_k|\le a_k$$
for all $1\le k\le m.$ Then:

{\rm (i)} There exists an $\go\in\C,\;|\go|=1$ so that for all $1\le k\le n,$
$$\ga_k=\go a_k.$$

{\rm (ii)}. If $\ga_j$ is real for some $1\le j\le m,$ then $\ga_k=a_k$ (thus is real) for all $k.$
\end{lemma}
\begin{proof}
By assumptions,
$$\sum_{k=1}^m a_k= |\sum_{k=1}^m \ga_k|\le \sum_{k=1}^m |\ga_k|\le \sum_{k=1}^m a_k.$$
Hence both inequalities are equalities. The last equality and the second assumption implies that $|\ga_k|=a_k$ or equivalently, $\ga_k=\go_ka_k$ where $|\go_k|=1.$  We wish to show that $\go_k=\go_1$ for all $k.$

Well, the triangle inequality for complex numbers implies that for any $k$ there exists a non-negative real number $r_k$ so that $\ga_k=r_k\ga_1,$  that is,
\begin{equation}\label{ak}\go_ka_k=r_k\go_1a_1.\end{equation}
Since $|\go_k|=|\go_1|=1,$ we have,
$r_k=\frac{a_k}{a_1}.$ Substituting in \eqref{ak}, we get $\go_k=\go_1$ for all $1\le k\le m.$

(ii) If $\ga_j$ is real for some $j,$ then  $\go_j=1$ hence $\go_1=1$ and all $\go_k$ equal $1$ as well and the result follows.
\end{proof}

We show:
\begin{theorem}\label{maxgeneral}
Let $H$ be a semisimple Hopf algebra over $\C$ with a character table $(\xi_{ij})$ where $\xi_{ij}=\langle\chi_i,\eta_j\rangle $ and   $f_jR(H)$  an irreducible right $R(H)$ module of dimension $m_j.$  Then

{\rm (i)} In each entry of the $i$-th row of the character table we have:
$$|\langle\chi_i,\eta_j\rangle  |\le m_j\langle\chi_i,1\rangle .$$

{\rm (ii)} Equality holds if and only if $r_{\chi_i}$ acts on $f_jR(H)$ as $\ga_i\Id_{f_jR(H)},$ where $|\ga_i|=\langle \chi_i,1\rangle.$

{\rm (iii)}  $\langle\chi_i,\eta_j\rangle = m_j\langle\chi_i,1\rangle $ if and only if $r_{\chi_i}$ acts on $f_jR(H)$ as $\langle\chi_i,1\rangle \Id_{f_jR(H)}.$
\end{theorem}
\begin{proof}
(i) By \eqref{ev}  each eigenvalue $\ga_{k},$ of the action of $r_{\chi_i}$ on  $f_jR(H),$   satisfies $|\ga_{k}|\le\langle\chi_i,1\rangle,$ for all $\,1\le k\le m_j.$  By Theorem \ref{muf}, $\left\langle \chi_i,\eta_j\right\rangle$ is the trace of this action, hence,
$$|\left\langle \chi_i,\eta_j\right\rangle|=|\sum_{k=1}^{m_j}\ga_{k}|\le\sum_{k=1}^{m_j}|\ga_{k}|\le m_j\langle\chi_i,1\rangle.$$

(ii) If equality holds in (i), then by applying Lemma \ref{complex} with $m=m_j$ and $a_k=\langle\chi_i,1\rangle$ for all $k,$   we obtain that all $\ga_k$ are equal and $|\ga_k|=\langle\chi_i,1\rangle.$  Hence the result follows. The converse is obvious.

(iii) If $\langle  \chi_i,\eta_j\rangle  =m_j\langle  \chi_i,1\rangle$ then by part (ii) it follows that $r_{\chi_i}$ acts on $f_jR(H)$ as $\ga\Id_{f_jR(H)}.$ Thus by Theorem \ref{muf}, $\langle  \chi_i,\eta_j\rangle  =\ga m_j,$ hence by assumption $\ga=\langle  \chi_i,\eta_j\rangle.$ The converse is obvious again.
\end{proof}

\begin{rema}\label{iso2}
Observe that  the theorem above is independent on the choice of $f_j\in F_jR(H).$  This follows for part (i) from Remark \ref{iso} and for parts (ii) and (iii) from the right $R(H)$-module isomorphism between $f_jR(H)$ and $f_{j'}R(H)$ for any other primitive idempotent on
\end{rema}

\section{Generalized kernels}
In this section we generalize an important normal subgroup that occurs in group theory. Let $G$ be a finite group and $V$ a representation of $G$  with associatedcharacter $\chi=\chi_V.$ Define
$$Z(\chi)=\{g\in G\,|\,|\langle\chi,g\rangle|= \langle\chi,1\rangle\}.$$
It is shown (see e.g. \cite{is}) that
$$g\in Z(\chi)\Leftrightarrow g\cdot v=\go_g v, \;\forall v\in V.$$
Note nthat $\go_g$ is a root of unity. In this sense $Z(\chi)$ is a `semi kernel' of $V.$
Based on this we define the {\bf $\go$-left kernel} as follows. Let $V$ be a representation of $H,$ set:
\begin{equation}\label{twist}  {\rm LKer}^{\go}_V=\{h\in H|h_1\ot h_2\cdot v=\go h\ot v\quad\forall v\in V,\;|\go|=1\}\end{equation}

When $\go=1$ then  ${\rm LKer}^{\go}_V$ boils down to the left kernel defined  in \cite{bu2} as follows:
\begin{equation}\label{lker}{\rm LKer}_V=\{h\in H\,|\,\sum h_1\ot h_2\cdot v= h\ot v,\,\forall v\in V\}.\end{equation}

Note that
\begin{equation}\label{kepsilon}{\rm LKer}_V=H\Leftrightarrow \chi_V=\gep.\end{equation}

Moreover, If $A$ a left coideal of $H$  then:
\begin{equation}\label{epsilon}A\subset{\rm LKer_V}\Leftrightarrow a\cdot v =\langle  \gep,a\rangle   v,\;\forall a\in A,\,v\in V.\end{equation}

We show:
\begin{lemma}\label{twistl}
Let $V$ be a representation of $H.$

\medskip{\rm (i)} If $0\ne {\rm LKer}^{\go}_V$   then it is a non-zero left $D(H)$-module, that is, it is a left coideal stable under the left adjoint action of $H.$

\medskip {\rm (ii)} $({\rm LKer}^{\go}_V)({\rm LKer}^{\go'}_V)\subset {\rm LKer}^{\go\go'}_V.$

\medskip{\rm (iii)} If  $\go^l=1$ then ${\rm LKer}^{\go}_V\subset {\rm LKer}_{V^{\ot l}}.$

\medskip{\rm (iv)} Set $Lz_{\chi}=\bigoplus_{\go}{\rm LKer}^{\go}_V.$ Then $Lz_{\chi}$ is a normal left coideal subalgebra of $H.$
\end{lemma}

\begin{proof}
\noin(i) The proof that ${\rm LKer}^{\go}_V$  is a left coideal stable under the left adjoint action of $H$ is exactly as  for these properties for ${\rm LKer}_V$ as shown in [2.1.3]\cite{bu2}.

\medskip\noin(ii) This follows by a direct computation.

\medskip\noin(iii)  Observe that ${\rm LKer}_{V_i}^{\go}$ is a left coideal so that each $h\in {\rm LKer}_{V_i}^{\go}$ acts on $V_i$ as $\go\left\langle\gep, h\right\rangle\Id.$ Hence we have for all $h\in {\rm LKer}_{V_i}^{\go},\;v_1,\dots,v_l\in V,$
$$h\cdot(v_1\ot\cdots\ot v_l)=\sum h_1\cdot v_1\ot\cdots\ot h_l\cdot v_l=\go^l\left\langle\gep, h\right\rangle(v_1\ot\cdots\ot v_l).$$ By \eqref{epsilon}, ${\rm LKer}^{\go}_V\subset {\rm LKer}_{V^{\ot l}}.$

\medskip\noin(iv) Follows from parts (i) and (ii).
\end{proof}

Since  $\go$-left kernels are left $D(H)$-modules, and since conjugacy classes are the irreducible left $D(H)$-modules it follows that  $\go$-left kernels are sums of conjugacy classes. In what follows we discuss further how they are related.

The following is the essential step in obtaining left kernels from the character table, generalizing the process of recovering normal subgroups of a finite group $G$ from its character table.
 Denote by ${\mathfrak C}^j$ the following left coideal of $H$ which is stable under the left adjoint action:
$${\mathfrak C}^j=\bigoplus_{k=1}^{m_j}{\mathfrak C}_{kj}=\gL\leftharpoonup F_jH^*.$$
Note that if $H$ is almost cocommutative then ${\mathfrak C}^j={\mathfrak C}_j,$ where ${\mathfrak C}_j$ is as defined in \eqref{fj0}.
Also note that $$\eta_j\in {\mathfrak C}^j\quad {\rm and}\quad\eta_j\leftharpoonup H^*={\mathfrak C}^j.$$

For a central  idempotent $F_j$ of $R(H)$  set
\begin{equation}\label{if}I_{F_j}=\{\chi|\chi F_j =\go_{\chi}\langle \chi,1\rangle F_j,\;\go_{\chi}\in\C,\,|\go_{\chi}|=1\},\end{equation}
where $\chi$ is any character of $H.$
Based on Theorem \ref{maxgeneral} we obtain the following:
\begin{proposition}\label{standard}
Let $H$ be a semisimple Hopf algebra over  $\C,\;F_j$ be a central primitive idempotent of $R(H),\;f_j$ any primitive idempotent of $R(H)$ so that $f_jF_j=f_j$ and $\dim f_jR(H)=m_j.$   Then:

{\rm (i)}   For any irreducible character $\chi_i,$
$$\chi_i\in I_{F_j}\Leftrightarrow |\langle \chi_i,\eta_j\rangle|=m_j \langle \chi_i,1\rangle.$$

{\rm (ii)} If $\chi\in I_{F_j}$ then so does each of its irreducible constituents $\chi_k,$ and $\go_{\chi_k}=\go_{\chi}.$

{\rm (iii)} For each $\chi\in I_F,\,\go_{\chi}$ is a root of unity.
\end{proposition}
\begin{proof}
(i)  Assume $\chi_i\in I_{F_j}.$ Set $\ga_i=\go_{\chi_i}\langle \chi_i,1\rangle.$ Since $\chi_i F_j=\ga_i F_j$ it follows that for all $p\in R(H),$
$$f_jp\chi_i=f_jpF_j\chi_i=\ga_i f_jp.$$
That is, $r_{\chi_i}$ acts on $f_jR(H)$ as $\ga_i\Id.$ This implies by Theorem \ref{maxgeneral}(ii)  that $|\langle \chi_i,\eta_j\rangle|=m_j \langle \chi_i,1\rangle.$ .

Conversely, if $|\langle\chi_i,\eta_j\rangle|=m_j \langle \chi_i,1\rangle$ then   Theorem \ref{maxgeneral}(ii) and Remark \ref{iso2} imply that $r_{\chi_i}$ acts on $f_jR(H)$ as $\ga_i\Id$ for all primitive idempotents  $f_j$ of  $R(H)$ belonging to $F_j.$  Since $F_j$ is their sum  it follows that $i\in I_{F_j}.$

(ii)   if $\chi\in I_{F_j}$ then as in the proof of (i), $r_{\chi}$ acts on $f_jR(H)$ as $\go_{\chi}\langle \chi,1\rangle\Id.$ By Theorem \ref{muf}, $\langle \eta_j,\chi\rangle $ is the trace of this action, hence
\begin{equation}\label{mj} \langle \chi,\eta_j\rangle =m_j \go_{\chi}\langle \chi,1\rangle.\end{equation}
Assume $\chi=\sum_{k=1}^n n_k\chi_k,\;n_k>0,$ then by \eqref{mj},
$$ m_j\sum_{k=1}^n n_k\langle\chi_k,1\rangle=   m_j \langle \chi,1\rangle=
m_j|\go_{\chi}\langle \chi,1\rangle|= |\langle \chi,\eta_j\rangle|=|\sum_{k=1}^n n_k \langle \chi_k,\eta_j\rangle|.$$

Apply Lemma \ref{complex} with $\ga_k=n_k \langle \chi_k,\eta_j\rangle$ and $a_k= n_km_j\langle \chi_k,1\rangle,$
to get  $\langle \chi_k,\eta_j\rangle= \go_{\chi} m_j \langle \chi_k,1\rangle$ for all $1\le k\le n.$  By part (i), $\chi_k\in I_{F_j}$ and since  $\chi=\sum_{k=1}^n n_k\chi_k$ it follows that $\go=\go_{\chi}.$

(iii) Take $\chi\in I_{F_j}.$ By \cite[Prop. 5(3)]{nr2}, $\gep$ is a constituent of $\chi^n$ for some $n\ge 1.$ By part (ii), $\go_{\chi^n}=\go_{\gep}=1.$ But $\go_{\chi^n}=(\go_{\chi})^n$ hence $\go_{\chi},$ is a root of unity.
\end{proof}

Let $B_0$ be a Hopf subalgebra of a semisimple Hopf algebra $B.$ In \cite[Prop.18]{nr2} the following equivalence relation was defined on simple coalgebras of $H^*:$
\begin{equation}\label{b}C_k\equiv_{B_0} C_{k'}\Leftrightarrow B_0C_k\subset C_{k'}.\end{equation}
Let
$$B_k=\bigoplus_{C_{k'}\equiv_{B_0} C_k}C_{k'}.$$

 Then $B_0B_k=B_k$ and $B=\bigoplus_kB_k.$ By the proof of \cite[Prop.18]{nr2} one can check that the equivalence relation on simple subcoalgebras is equivalent to the following equivalence relation defined on irreducible characters inside $B:$
\begin{equation}\label{equi}
\chi_i\equiv_{B_0}\chi_j\Leftrightarrow \gl_{B_0} \frac{\chi_i}{d_i}=\gl_{B_0} \frac{\chi_j}{d_j}
\end{equation}
Denote by $[\chi_k]$ the equivalence class of $\chi_k$ and set $y_k$
\begin{equation}\label{lc} y_k=\sum_{\chi_{k'}\in [\chi_k]}\langle\chi_{k'},1\rangle\chi_{k'}.\end{equation}
Then $y_0=\gl_{B_0}$ and $\sum_{k}y_k=\gl_B$ is an integral for $B$ and for each $k,\,\left\langle y_k,1\right\rangle=\dim B_k.$
We are ready to prove the following useful proposition:
\begin{proposition}\label{omega0}
Let $H$ be a semisimple Hopf algebra over an algebraicly closed field $k$ of characteristic $0,\;\chi$ an irreducible character and $F_j$ a central primitive idempotent of $R(H).$ Let $B$ be the Hopf algebra generated by $\chi.$ Assume $\chi\in I_{F_j}.$ Then the following hold:

{\rm (i)} All irreducible characters inside $B$ belong to $I_{F_j}.$

{\rm (ii)} Let $B_0$ be the coalgebra generated by all irreducible characters  $\chi_k\in B$ so that $F_j\chi_k=\langle\chi_k,1\rangle F_j.$ Then  $B_0$ is a Hopf subalgebra of $B.$
\medskip

{\rm (iii)}  Inside $B,$ consider the equivalence relation defined in \eqref{b}. Set $\go=\go_{\chi}.$ We can arrange $B_m$ so that $\go_{\chi_k}=\go^m$
for all irreducible characters $\chi_k\in B_m.$
\medskip

{\rm (iv)}  $B=\bigoplus_{m=0}^{t-1}B_m$ is a graded algebra and $\go$ is a root of unity.  The number $t$ is the least $n$ so that $\gep$ is a constituent of $\chi.$
\medskip

{\rm (v)}  $\dim B_m=\dim B_0$ for all $0\le m\le t-1.$
\medskip
\end{proposition}
\begin{proof}
\noin(i). Each irreducible character inside $B$ is a constituent of $\chi^l$ for some $l\ge 0.$ But $\chi^l\in I_{F_j}$ where $\go_{\chi^l}=\go^l$ hence the result follows from Proposition \ref{standard}(ii).

\medskip\noin(ii). Since $\go_{\chi}\go_{\chi'}=\go_{\chi\chi'}$ for all $\chi,\chi'\in B,$ it follows by Proposition \ref{standard}(ii) that $B_0$ is an algebra, and thus a Hopf subalgebra.

\medskip\noin(iii). First note that for all $\chi_k,\chi_{k'}\in B_k,$ we have  by \eqref{equi} $\go_{\chi_k}=\go_{\chi_{k'}}.$  We wish to show that to each $B_k$ corresponds a unique $\go_k.$ That is, if  $\chi_k\in B_k,\,\chi_t\in B_t$ and $\go_{\chi_k}=\go_{\chi_t},$ then $B_k=B_t.$

Since $\gep$ is a constituent of $\chi_ts(\chi_t)$ it follows that $\go_{s(\chi_t)}=\go_{\chi_t}\minus.$ Thus $\chi_ks(\chi_t)\in B_0.$ But then
$$\chi_ks(\chi_t)\chi_t\in B_0\chi_t\subset B_t.$$

On the other hand, since $\gep$ is a constituent of $s(\chi_t)\chi_t$ it follows that $\chi_k$ is a constituent of $\chi_ks(\chi_t)\chi_t\in B_t.$ Since $B_t$ is a sub coalgebra it follows that  $\chi_k\in B_t.$ Thus $B_k=B_t.$

\medskip\noin(iv). This follows now after arranging the subcpalgebras $B_m$ with respect to powers of $\chi.$

\medskip\noin(v).
Take $b\in B_m,$ then we have:
$$\langle b,1\rangle\sum_{l\ge 0} y_l=\langle b,1\rangle\gl_{B}=b\gl_{B} =b\sum_{l\ge 0} y_l.$$
Since $by_0\in B_{m}$ and the sum is direct, it follows that
\begin{equation}\label{alpha3}by_0=\langle b,1\rangle y_{m}.\end{equation}
Take $b\in B_m$ so that $\langle b,1\rangle\ne 0,$ (e.g. $b=\chi^m$), then applying $1$ to both sides of \eqref{alpha3} yields
\begin{equation}\label{alpha4}\dim B_{0}=\dim B_{m}\end{equation}
 for all $0\le m\le t-1.$
Since $\chi$ is the only irreducible character in $B_1$ it follows that $\dim B_1=\langle\chi,1\rangle^2,$ hence the result follows.
\end{proof}
As a corollary we obtain:
\begin{coro}\label{main}
With notations as in Proposition \ref{omega0}, we have:  $F_jb=\go^m\langle b,1\rangle F_j$ for all $b\in B_m.$
\end{coro}
\begin{proof} Let $b\in B_m.$ By \eqref{alpha3} and \eqref{alpha4} we have:
$$bF_j=b\frac{y_0}{\dim B_0}F_j=\frac{1}{\dim B_0}\langle b,1\rangle y_mF_j=\frac{1}{\dim B_0}\langle b,1\rangle \go^m\dim B_m F_j=\langle b,1\rangle  \go^m F_j.$$
\end{proof}

We can prove now our main theorem.
\begin{theorem}\label{connection3}
Let $H$ be a semisimple Hopf algebra over an algebraicly closed field $k$ of characteristic $0,\;\chi_i$ an irreducible character and $F_j$ a central primitive idempotent of $R(H).$ Then the following are equivalent:

{\rm (i)} $F_j\chi_i=\go_i\langle \chi_i,1\rangle  F_j,\,|\go_i|=1.$

\medskip{\rm (ii)} $\langle \chi_i,\eta_j \rangle = \go_im_j\langle \chi_i,1 \rangle,\,|\go_i|=1 .$

\medskip{\rm (iii)} ${\mathfrak C}^{j}\subset {\rm LKer}^{\go_i}_{V_i}.$

\medskip Moreover, if the equivalent conditions hold then $\go_i$ is a root of unity of order $t,$ where $t$ divides $\dim H.$
\end{theorem}
\begin{proof}
$(i)\Leftrightarrow (ii)$ is Theorem \ref{standard}(i).

$(iii)\Rightarrow (ii).$ If ${\mathfrak C}^{j}\subset {\rm LKer}^{\go_i}_{V_i}$ then in particular $\eta_j\in {\mathfrak C}^j$ acts as $\go_im_j$ on $V_i,$ hence  (ii) follows.

$(i)\Rightarrow (iii).$  By  \eqref{epsilon} and since ${\mathfrak C}^{j}$ is a left coideal,  it is enough to show , that $a\cdot v=\go_i \langle \gep,a\rangle v$ for all $a\in {\mathfrak C}^{j},\,v\in V_i.$

As a left $H$-module, $V_i$ is a right $H^*$-comodule, with $\rho(v)=\sum v_0\ot v_1$ for all $v\in V_i.$  Take $\chi=\chi_i$ in Proposition \ref{omega0}, then $sp_{\C}\{v_1\}\subset B_1$ in the same proposition.

Let $a=(\gL\leftharpoonup F_jp)\in {\mathfrak C}^j,\;v\in V.$ Then:
\begin{eqnarray*}\lefteqn{a\cdot v=}\\
&=&\sum_v\langle\gL\leftharpoonup F_jp,v_1\rangle v_0\\
&=&\sum_v\langle\gL,v_1F_jp\rangle v_0\qquad\text{(since $\gL$ is cocommutative)}\\
&=&\sum_v\langle v_1,1\rangle\go_i\langle \gL, F_jp\rangle v_0\qquad\text{(by Corollary \ref{main})}\\
&=&\go_i\langle\gep,a\rangle v.\end{eqnarray*}

The fact that $\go$ is a root of unity follows from Proposition \ref{standard}(iii). By Proposition \ref{omega0}(iv)-(v), $(\dim B)=t(\dim B_0).$ Since $(\dim B)|(\dim H)$ the last result follows.
\end{proof}

As a corollary we obtain:
\begin{theorem}\label{connection}
Let $H$ be a semisimple Hopf algebra over $\C,\;B$ the Hopf subalgebra of $H^*$ generated by $\chi_i,$ and $F_j$ a central primitive idempotent of $R(H).$ Then the following are equivalent:

\noin{\rm (i)} $F_j\chi_i=\left\langle \chi_i,1 \right\rangle F_j$

\medskip\noin{\rm (ii)}  $\left\langle \chi_i,\eta_j \right\rangle =m_j\left\langle \chi_i,1 \right\rangle .$

\medskip\noin{\rm (iii)} ${\mathfrak C}^{j}\subset {\rm LKer}_{V_i}.$
\end{theorem}

\medskip In the next section (Lemma \ref{seno1}), we show that $\go$-left kernels are abundant for factorizable Hopf algebras.

\medskip

\medskip
We generalize the notion of faithful characters as follows:
\begin{definition} A character $\chi$ associated with a representation $V$ will be called {\it faithful} if ${\rm LKer_V}=k1.$ \end{definition}
It was proved:

\medskip\noin{\bf Generalized Burnside-Brauer Theorem} \cite[4.2.1]{bu2}:
Let $H$ be a semisimple Hopf algebra over an algebraically closed field $k$ of characteristic $0,$  $V$ a representation of $H$ with associated character $\chi$
and $N= {\rm LKer}_{V}.$ Then the irreducible characters of $\ol{H}=H/({N^+}H)$
are precisely all the irreducible constituents of $\chi^n,\,n\ge 0.$

\medskip
When the semisimple Hopf algebra $H$ is an almost cocommutative, or equivalently, when it has a  commutative character algebra, we give  an explicit bound to the powers of $\chi$ appearing above, thus having a {\bf full} analogue of the Burnside-Brauer theorem.

Define the value set
$$\chi(H)=\{\langle  \chi,\eta_j\rangle  \,|\;\text{all }\eta_j\}.$$
We show:
\begin{theorem}\label{burnside}
 Let $H$ be a semisimple Hopf algebra such that $R(H)$ is commutative and let $\chi$ be a faithful character of $H.$ Set $$t= \text{number of distinct values in }\, \chi(H).$$
 Then each character $\mu$ of $H$ appears with positive multiplicity in at least one of $\{\gep,\chi,\chi^2,\dots,\chi^{t-1}\}.$
\end{theorem}
\begin{proof}
Since $\chi$ is faithful it follows from the generalized Burnside-Brauer theorem that all irreducible characters of $H$ appear as constituents of powers of $\chi.$
We will show that the $k$-span of $\{\gep,\chi,\chi^2,\dots,\chi^{t-1}\}$ already contains all powers of $\chi.$  Let $\{\ga_0,\dots \ga_{t-1}\}$ be the different values of $\langle  \chi,\eta_j\rangle  $ where $\ga_0=\langle  \chi,1\rangle  .$ Let
$$T_i=\sum_jF_{i_j},$$
where by \eqref{almost}, $\{F_{i_j}\}$ are all the central idempotents satisfying $\chi F_{i_j}=\ga_iF_{i_j}.$ By assumption $T_0=d\minus\gl.$ By \cite{cw3} we have:
$$\chi=\sum_{i=0}^{t-1}\ga_iT_i$$ Since $\{T_i\}$ are orthogonal idempotents we get
$$\chi^l=\sum_{i=0}^{t-1}\ga_i^lT_i$$
for all $0\le l\le t-1.$ The Vandermonde matrix $(\ga_i^l)$ is invertible, hence each $T_i$ is a linear combination of $\{\gep,\chi,\chi^2,\dots,\chi^{t-1}\}.$ Since any power of $\chi$ is a linear combination of the $T_i$'s, it follows that the algebra generated by $\chi$ inside $R(H)$ is spanned over $k$ by $\{\gep,\chi,\chi^2,\dots,\chi^{t-1}\}.$
\end{proof}


\medskip

As a corollary we obtain that the only Hopf algebras for which there exist a $1$-dimensional representation with a faithful character satisfy a stronger condition then being almost cocommutative. They are in fact both commutative and cocommutatie.
\begin{coro}\label{1dim}
Let $H$ be a semisimple Hopf algebra over an algebraically closed field of characteristic $0.$  Then the following are equivalent:

{\rm (i)} There exists  a $1$-dimensional representation of $H$ with a faithful character.

{\rm (ii)} $R(H)=k\left\langle \gs\right\rangle,$ where $\gs$ is a grouplike element of $H^*.$

{\rm (iii)} $H=(k\left\langle \gs\right\rangle)^*,\gs$ is a grouplike element of $H^*.$
\end{coro}
\begin{proof}
If $H=k1$ then we are done. So, assume $H\ne k1.$

(i)$\Rightarrow$ (ii).  Since $V$ is $1$-dimensional, $\gep\ne\gs=\chi_V$ is a grouplike element. By Theorem \ref{burnside}, each irreducible character $\chi$ is a constituent of $\gs^j,\,j\ge 0.$ Since $\gs^j$ is an irreducible character it follows that $\chi=\gs^j.$

(ii)$\Rightarrow$ (iii). Since all irreducible representations of $H$ are $1$-dimensional it follows that $H$ is a commutative algebra, hence $H^*$ is cocommutative, and so $H^*=R(H).$

(iii)$\Rightarrow$ (i). Take $V=k\gs,$ then $V$ is an $H$-module under left {\it hit} and $\chi_V=\gs$ is faithful.
\end{proof}

\section{Applications}

In this section we prove properties of semisimple Hopf algebras over $\C$ related to their character table with particular emphasis on Burnside's $p^aq^b$ theorem. It is known that this theorem fails if we use  a sequence of normal Hopf subalgebras to define solvability. (A counter example of dimension $36$ is given in \cite{gn}).

A categorical version of Burnside's  theorem is given in \cite{eno}. They prove that any fusion category of Frobenius-Perron dimension $p^aq^b$ is solvable. However, it is not clear how to interpret all the categorical notions to Hopf algebras.
The discussion in this section indicates that a sequence of normal left coideal subalgebras should be used to define solvability in this case. We prove some of the results in \cite{eno} for semisimple Hopf algebras of dimension $p^aq^b$ by using the algebraic tools we developed in the previous sections.

\medskip We start by proving a Hopf algebra analogue of the following:

\medskip \noin{\bf A theorem of Burnside}(see e.g. \cite[Th. 3.8]{is}): If G is a non-abelian simple group then $\{1\}$ is the only conjugacy class of $G$ which has prime power order.

Note that $1$ is considered as a prime power as well.
The discussion goes further and we prove some structural properties mainly of factorizable Hopf algebras.

\medskip
Recall, if $(H,R)$ is a finite dimensional quasitriangular
Hopf algebra then the maps $f_R:H^{*cop}\rightarrow H,$
defined by $f_R(p)=\langle p,R^1\rangle  R^2$ and ${f^*_R}:H^{*op}\mapsto
H,$ defined by ${f^*_R}(p)=\langle p,R^2\rangle  R^1$ are Hopf algebras
maps. Then $Q=R^{21}R$ and the Drinfeld map $f_Q$ is given by $f_Q={f^*_R}*f_R.$ When $(H,R)$ is a  quasitriangular Hopf algebra then $R(H)$ is necessarily commutative. If $H$ is also semisimple then the Drinfeld map $f_Q$ is an algebra isomorphism between $R(H)$ and $Z(H).$

The $S$-matrix for $(H,R)$ is defined by:
\begin{equation}\label{sij}s_{ij}=\langle \chi_i,f_Q(\chi_j)\rangle  .\end{equation}
The quasitriangular Hopf algebra $(H,R)$ is {\it factorizable} if $f_Q$ is a
monomorphism, or equivalently, if its $S$-matrix is invertible.

Assume $(H,R)$ is factorizable. Reorder the set $\{F_j\}$ so that
$$f_Q(F_j)=E_j$$
for all $1\le j\le m.$
It follows that $\dim (F_jH^*)=d_j^2,$ where $d_j= \langle\chi_j,1\rangle.$ Recall \cite{cw3} that, after reordering the central primitive idempotents,  we have:
\begin{equation}\label{chici}f_Q(\chi_j)=\frac{1}{d_j}C_j=d_j\eta_j.\end{equation}
The last equality follows since $\dim(F_jH^*)=\dim(E_jH)=d_j^2.$
It follows that the $S$-matrix satisfies
\begin{equation}\label{smatrix}s_{ij}=d_j\xi_{ij}. \end{equation}

In what follows we discuss further properties of $f_Q.$
\begin{lemma}\label{fq1}
Let $(H,R)$ be a quasitriangular semisimple Hopf algebra over $k.$ Then $f_Q$ maps -coalgebras of $H^*$ to left coideals  stable under the adjoint action  of $H$ and Hopf subalgebras of $H^*$ to normal left coideal subalgebras of $H.$
\end{lemma}
\begin{proof}
Let $C$ be a subcoalgebra of $H^*,\,c\in C.$ By \cite[Prop.2.5(3)]{cw2},
$$f_Q(c)\leftharpoonup p=f_Q\left(f_R^*(p_2)\rightharpoonup c\leftharpoonup f_R(p_1)\right)\in f_Q(C),$$
for all $p\in H^*.$  Since $C$ is a coalgebra we have that $f_Q(C)$ is a left coideal of $H.$  It is stable under the adjoint action of $H$ by \cite[Prop.2.5(5)]{cw2}.

If $C$ is a Hopf subalgebra of $H^*$ then all that is left to show is that $f_Q(C)$ is an algebra. Let $x,y\in C,$ then by the definition of $f_Q$ we have,
$$f_Q(x)f_Q(y)={f^*_R}(x_1)f_R(x_2)f_Q(y)= {f^*_R}(x_1)\left[f_R(x_3)\ad f_Q(y)\right]f_R(x_2),$$
where the last equation follows since $f_R$ is an anticoalgebra map. Since $f_Q(C)$ is stable under the adjoint action of $H,$ the mid term above belongs to $f_Q(C).$ The result follows now from \cite[Prop.2.5(2)]{cw2}.
\end{proof}
If $(H,R)$ is factorizable more can be said:
\begin{lemma}\label{subn}
Let $(H,R)$ be a factorizable semisimple Hopf algebra over $k.$ Then,

{\rm (i)} For all $0\le j\le n-1,$ let $R_{V_j}$ be the coalgebra generated by $\chi_j,$ then $f_Q(R_{V_j})={\mathfrak C}_j.$

\medskip {\rm (ii)} Let $N$ be a left coideal of $H.$ If a character $\chi \in f_Q\minus(N)$ then  the same hold for all its constituents.

\medskip {\rm (iii)}   Let $N$ be  a normal left coideal subalgebra of $H$. Then $f_Q\minus(N)$ is a Hopf subalgebra of $H^*.$
\end{lemma}
\begin{proof}
(i) By
 Lemma \ref{fq1}, $f_Q(R_{V_j})$ is a left coideal  of $H.$ By \eqref{chici},    $\eta_j\in f_Q(R_{V_j}).$ Since ${\mathfrak C}_j$ is the left coideal generated by $\eta_j$ we have ${\mathfrak C}_j\subset f_Q(R_{V_j}).$ But $f_Q$ is injective  and $\dim  {\mathfrak C}_j=d_j^2=\dim  R_{V_j},$ hence equality holds.

(ii) Let $\chi=\sum_{j\in J}m_j\chi_j,\,m_j\in \Z^+.$ By \eqref{chici}, $f_Q(\chi)=\sum_{j\in J} d_jm_j\eta_j\in N.$ For each $j\in J,$ we have $d_jm_j\eta_j=f_Q(\chi) \leftharpoonup F_j\in N.$ Thus $\chi_j\in f_Q\minus(N).$

(iii) Since $N$ is a normal left coideal subalgebra, it is a left $D(H)$-module, hence by \eqref{nik}, there exists a set $J\subset \{0,\dots,n-1\}$ so that $N=\oplus_{j\in J} {\mathfrak C}_j.$  By (i) $f_Q\minus(N)=\oplus_{j\in J}(R_{V_j})$ hence it is a coalgebra.  Let $\chi_k,\,\chi_l\in f_Q\minus(N).$ Since $f_Q$ is multiplicative on $R(H)$ and $N$ is an algebra we have:
$$f_Q(\chi_k\chi_l)=f_Q(\chi_k)f_Q(\chi_l)\in N.$$
Hence $\chi_k\chi_l\in f_Q\minus(N).$ By (ii) each constituent of $\chi_k\chi_l\in f_Q\minus(N)$ as well.

We want to show that $f_Q\minus(N)$ is an algebra. Let $a\in R_{V_j},\,b\in R_{V_k},$ where $R_{V_j},\,R_{V_k}\subset f_Q\minus (N).$ In particular $\chi_j,\,\chi_k\in N.$ Assume $\chi_j\chi_k=\sum m_i\chi_i,$ then by (ii), each $\chi_i\in f_Q\minus(N)$ and thus by (i), each $R_{V_i}\subset N.$ W have now,
$$ab\in R_{V_j}R_{V_k}=R_{V_j\ot V_k}=\oplus R_{V_i}\subset f_Q\minus(N).$$
\end{proof}

We show now how $\go$-left kernels arise for factorizabble Hopf algebras.

\begin{lemma}\label{seno1}Let $(H,R)$ be a factorizable semisimple Hopf algebra over $\C,$ then

{\rm (i)} $\frac{d_j\langle  \chi_i,\eta_j\rangle  }{d_i}$ is an algebraic integer for all $i,j.$

\medskip{\rm (ii)} If $\gcd(d_i,d_j)=1$ and $\langle  \chi_i,\eta_j\rangle  \ne 0,$ then $\go_i=\frac{\langle  \chi_i,\eta_j\rangle  }{d_i}$ is  a root of $1.$  Thus
 ${\mathfrak C}^{j}\subset {\rm LKer}^{\go_i}_{V_i}.$
\end{lemma}
\begin{proof}
(i) Since $\frac{s_{ij}}{d_i}$ is known to be an algebraic integer, the result follows from \eqref{smatrix}.

(ii) Assume $\langle  \chi_i,\eta_j\rangle  \ne 0.$ Since $\gcd(d_i,d_j)=1,$ we may choose integers $u,v$ so that $ud_i+vd_j=1.$ Multiplying both sides by $\frac{\langle  \chi_i,\eta_j\rangle  }{d_i},$ we get
$$\frac{\langle  \chi_i,\eta_j\rangle  }{d_i}=u\langle  \chi_i,\eta_j\rangle  +v\frac{\langle  \chi_i,\eta_j\rangle  d_j}{d_i}$$
The first summand on the right hand side is an algebraic integer, which we show to be a root of unity. by Corollary \ref{algebraic}, while the second one is  by the first part. It follows that $\frac{\langle  \chi_i,\eta_j\rangle  }{d_i}$ is an algebraic integer.

Recall, \cite{v},  that $Q$ has finite order, hence its action on $V_i\ot V_j$ is diagonalizable with eigenvalues which are roots of unity. Since $\chi_i\ot\chi_j$ is the trace of $H\ot H$ acting on $V_i\ot V_j$ it follows that
$$\langle  \chi_i\ot \chi_j,Q\rangle  =\sum_{k=1}^{d_id_j}\go_k,$$ where $\go_k$ is a root of unity.
Since the $S$-matrix is symmetric we have by \eqref{sij} and \eqref{smatrix},
$$\frac{\langle  \chi_i,\eta_j\rangle  }{d_i}=\frac{s_{ji}}{d_id_j}=\frac{\langle  \chi_i\ot \chi_j,Q\rangle  }{d_id_j}=\frac{1}{d_id_j}\sum_{k=1}^{d_id_j}\go_k$$
Since $\frac{\langle  \chi_i,\eta_j\rangle  }{d_i}$ is an algebraic integer, which is nonzero by assumption,  it follows by elementary Galois theory (see e.g \cite[3.8]{is}), that all $w_k$ are equal and thus the left handside is a root of unity.

The last part follows since  factorizable Hopf algebras are almost cocommutative, hence we have $m_j=1$ and ${\mathfrak C}_j={\mathfrak C}^j,$ the result follows from  Theorem \ref{connection}.
\end{proof}
\medskip

\medskip
The following  lemma arises from the  proof of Burnside $p^aq^b$ theorem for groups.
\begin{lemma}\label{proper0}
Let $H$ be a  semisimple Hopf algebra over $\C,\,q$ a prime number.  Then for each $0< j\le m-1$ there exists $0< i\le n-1$ so that $\langle  \chi_i,\eta_j\rangle  \ne 0$ and $q$ does not divide $d_i.$
\end{lemma}
\begin{proof}
By the orthogonality relations with respect to the first column, (Theorem \ref{ortho1}), we have,
\begin{equation}\label{e1}
0=\sum_m d_m\langle  \chi_m,\eta_j\rangle
\end{equation}

We claim that there exists $i>  0$ so that $q$ does not divide $d_i$ and $\langle  \chi_i,\eta_j\rangle  \ne 0.$ Indeed, let $S$ be the set of all $V_m$ so that $q|d_m$ and let $T$ be the set of all non-trivial $V_m$ so that $q$ does not divide $d_m.$
By \eqref{e1}, $$0=1+\sum_{V\in S}d_m\langle  \chi_m,\eta_j\rangle  +\sum_{V\in T}d_m\langle  \chi_m,\eta_j\rangle  $$

Set
$$a= \sum_{V\in S}\frac{d_m}{q}\langle  \chi_m,\eta_j\rangle  $$
Then $a$ is an algebraic integer and
$$0=1+qa+\sum_{V\in T}d_m\langle  \chi_m,\eta_j\rangle  .$$
If the last summand is $0,$ then $a$ is a rational algebraic integer, hence an integer which is impossible. Hence there exists $i$ so that $\langle  \chi_i,\eta_j\rangle  \ne 0$ and $q\not|d_i.$
\end{proof}
We are ready to prove a Hopf algebra analogue the theorem of Burnside for groups mentioned at the beginning of this section.
\begin{theorem}\label{isa}
Let $(H,R)$ be a non-commutative quasitriangular semisimple Hopf algebra over $\C$ whose only normal left coideal subalgebras are $\C$ and $H.$ Then $\C$ is the only  conjugacy class of $H$ with  a prime power dimension.
\end{theorem}
\begin{proof}
By Lemma \ref{fq1}, Im$f_Q$ is a normal left  coideal sbalgebra of $H,$ hence by assumption either (i) Im$f_Q=\C$ or (ii) Im$f_Q=H.$

\medskip\noin(i) If Im$f_Q=\C$ then $Q=1\ot 1$ hence $H$ is triangular. In this case the Drinfeld element $u$, is a central grouplike element of order $\le 2$ \cite{dr,r1}.

If $u\ne 1$ then $k\left\langle u\right\rangle$ is a normal left  coideal subalgebra hence equals $H,$ contradicting non-commutativity.

If $u=1$ then by \cite{g}, $H=(kG)^J$ where $J$ is a twist for $kG.$ By \cite[Cor.2.12]{cw3} $H$ has the same character table as $kG$ and thus has the same lattice of normal left  coideal subalgebras and by Remark \ref{cjcharacter}, the same dimensions of conjugacy classes. By the theorem of Burnside for groups, $G$ has no  conjugacy class with a prime power order, hence neither does $H.$

\medskip\noin (ii) If Im$f_Q=H$ then $H$ is factorizable. Let ${\mathcal C}_j,\,j>0,$ be any conjugacy class. If $\dim {\mathcal C}_j=1$ then $\eta_j$ is a central grouplike element of $H$ implying that the group algebra generated by $\eta_j$ is a normal left coideal subalgebra of $H.$ Non-commutativity of $H$  and the hypothesis of the theorem imply $\eta_j=\gep,$ a contradiction to $j>0.$

 Assume $\dim {\mathcal C}_j=q^l\ne 1.$     By  Lemma \ref{proper0} there exists $i> 0$ so that $V_i$ has dimension prime to $q$ and $\langle\chi_i,\eta_j\rangle\ne 0.$
Since $H$ is factorizable we have $\dim({\mathfrak C}_j)=d_j^2,$ and so $(d_i,d_j)=1.$ By Lemma \ref{seno1}(ii), there exists $\go,\,\go^l=1$ and ${\mathfrak C}_j\subset {\rm LKer}_{V_i}^{\go},$ which by Lemma \ref{twistl}(iii) is contained in ${\rm LKer}_{V_i^{\ot l}}.$ In particular $\C\ne {\rm LKer}_{V_i^{\ot l}}.$

Thus if $d_i\ne 1$ we have $\chi_{V_i^\ot l}\ne \gep$ hence by \eqref{kepsilon},
 ${\rm LKer}_{V_i^{\ot l}}\subsetneqq H,$ a contradiction.
If $d_i= 1$ then by  Corollary \ref{1dim}, $H$ must be commutative, a contradiction.
\end{proof}

\medskip The following results, culminating in Theorem \ref{sub},  follow directly from \cite{eno} who obtained it using categorical consideratins. We represent it here in algebraic terms as consequences of the previous discussion.
\begin{proposition}\label{proper}
Let $H$ be a quasitriangular semisimple Hopf algebra over $\C$ of dimension $p^aq^b,\,a+b> 1,\,p,q$ prime numbers. Then  $H$ has a proper non-trivial normal left  coideal subalgebra.
\end{proposition}
\begin{proof}
If $H$ is commutative then quasitriangularity implies that it is cocommutative, hence a group algebra. Since the order of the group is not prime, it has a normal subgroup. Hence $H=kG$ has a non-trivial normal left coideal subalgebra.

Assume $H$ is non-commutative and $a>0.$  We claim that there exists $j>  0$ so that $p$ does not divide $\dim{\mathfrak C}_j.$ This follows  from the fact that the dimension of each conjugacy class divides the dimension of $H$ (\cite{k,z}), and  their sum equals $\dim H.$ Since ${\mathfrak C}_0=\C,$ $p$  can not divide the dimensions of all other conjugacy classes.  That is, there exists $j>0$ so that ${\mathcal C}_j$ has a prime power dimension. By Theorem \ref{isa}, $H$ contains a proper non-trivial normal left  coideal subalgebra.
\end{proof}

\medskip

\begin{coro}\label{max}
Let $H$ be a semisimple quasitriangular Hopf algebra of dimension $p^aq^b,a+b>  0,$ over $\C.$ Then $H$ admits a nontrivial $1$-dimensional representation.
\end{coro}
\begin{proof}
Let $M\ne H$ be a maximal normal left  coideal subalgebra of $H.$ Then $\ol{H}=H/({M^+}H)$ has no non-trivial proper normal left  coideal subalgebras and $\dim \ol{H}=p^{a'}q^{b'},\,a'+b'>0.$ Proposition \ref{proper} implies now that  $a'+b'=1.$ Hence $\ol{H}\cong \Z_p$ and thus has a $1$-dimensional representation which is also an $H$-representation.
\end{proof}
We can show now:
\begin{proposition}\label{gl}\cite[Prop.9.9]{eno}.
Let $H$ be a semisimple Hopf algebra over $\C$ of dimension $p^aq^b,\,a+b>  0,\;p,q$ prime numbers. Then $H$ admits a non-trivial $1$-dimensional representation and a grouplike element $\gs\ne 1.$
\end{proposition}
\begin{proof}
If $a+b=1$ then $H$ is a group algebra of prime order and we are done. Otherwise, consider the factorizable Hopf algebra $D(H).$ By Corollary \ref{max} it has a non-trivial $1$-dimensional representation, hence $D(H)^*$ has a non-trivial grouplike element, which  by \cite{ra2},  has the form $\gs\ot\tau,$ where $\gs\in G(H),\;\tau\in G(H^*).$ If $\tau\ne\gep$ and $\gs\ne 1$ then we are done.

Assume now $\tau=\gep.$ By \cite{dr,ra2},
$\gs\ne 1$ is a central grouplike element of $H.$  Let $N$ be the central Hopf subalgebra of $H$ generated by $\gs.$  Then by the induction hypothesis  $H/HN^+$ has a $1$-dimensinal representation hence so does $H.$  The same proof, replacing $H$ with $H^*,$ works when $\gs=1.$

\end{proof}

These culminate in:
\begin{theorem}\label{sub}
Let $H$ be a semisimple Hopf algebra over $\C$ of dimension $p^aq^b,\,a+b>  0,\;p,q$ prime numbers. Then either $H$ is a group algebra of prime order, or  $H$ contains a normal left  coideal subalgebra $N$ so that $\C\subsetneqq N\subsetneqq H.$
\end{theorem}
\begin{proof}
By Proposition \ref{gl}, $H$ has a non-trivial $1$-dimensional representation $V,$ hence ${\rm LKer_V}\ne H$ is a normal left  coideal subalgebra of $H.$ If $V$ is not faithful then $\C\subsetneqq {\rm LKer_V}$ and we are done. Otherwise, by Corollary \ref{1dim}, $H$ is a  group algebra. If $a+b>  1$ then this group contains a normal subgroup $N,$ hence contains $kN,$ a normal left coideal subalgebra. If $a+b=1$ then $H$ is a group algebra of prime order.
\end{proof}

When $H$ is factorizable more can be said:
\begin{theorem}\label{hopfsub}
Let $H$ is a factorizable semisimple  Hopf algebra over \C  of dimension $p^aq^b,\,a+b>0,$ and let $N$ be a  normal left  coidal subalgebra of $H.$ Then $N$ contains a central grouplike element. In particular, any minimal normal left  coidal subalgebra of $H$ is a central Hopf subalgebra of $H$ of  prime order.
\end{theorem}
\begin{proof}

By Lemma \ref{subn} and Proposition \ref{gl}, the Hopf algebra $f_Q\minus(N)$ contains a non-trivial grouplike element. Since $f_Q$ maps grouplike elements of $H^*$ to central grouplike elements of $H$ we are done. The rest is straightforward.
\end{proof}


\end{document}